\documentclass[a4paper, 12pt]{amsart}

\usepackage{amssymb, amsmath, amsthm, latexsym}
\usepackage{amscd}
\usepackage{multirow}
\usepackage[varg]{txfonts}

\theoremstyle{plain}
\newtheorem{Theorem}{Theorem}[section]
\newtheorem{theorem0}{Thorem}

\newtheorem{Proposition}[Theorem]{Proposition}
\newtheorem{Proposition-Definition}[Theorem]{Proposition-Definition}
\newtheorem{Lemma}[Theorem]{Lemma}

\theoremstyle{definition} 
\newtheorem{Definition}[Theorem]{Definition}
\newtheorem{Remark}[Theorem]{Remark}
\newtheorem{Example}[Theorem]{Example}

\newtheorem{Notation}[Theorem]{Notation}

\def\A{{\mathcal{A}}}
\def\B{{\mathcal{B}}}
\def\E{{\mathcal{E}}}
\def\C{{\mathcal{C}}}

\def\V{{\mathcal{V}}}
\def\H{{\mathfrak{H}}}
\def\G{{\mathfrak{G}}}
\def\Z{{\varmathbb{Z}}}
\def\N{{\varmathbb{N}}}
\def\Q{{\varmathbb{Q}}} 
\def\R{{\varmathbb{R}}} 

\def\a{{\boldsymbol{a}}}
\def\b{{\boldsymbol{b}}}
\def\c{{\boldsymbol{c}}}
\def\e{{\boldsymbol{e}}}
\def\w{{\boldsymbol{w}}}
\def\r{{\boldsymbol{r}}}
\def\uu{{\boldsymbol{u}}}
\def\vv{{\boldsymbol{v}}}

\def\x{{\boldsymbol{x}}}
\def\y{{\boldsymbol{y}}}
\def\z{{\boldsymbol{z}}}

\def\initial{\mathop{\mathrm{in}}\nolimits}

\def\lift{\mathop{lift}\nolimits}
\def\lift{\mathop{\mathrm{lift}}\nolimits}
\def\Lift{\mathop{\mathrm{Lift}_\prec}\nolimits}

\def\max{\mathop{\mathrm{max}}\nolimits}
\def\Ker{\mathop{\mathrm{Ker}}\nolimits}
\def\deg{\mathop{\mathrm{deg}}\nolimits}

\def\lcm{\mathop{\mathrm{lcm}}\nolimits}
\def\Fiber{\mathop{\mathrm{Fiber}}\nolimits}

\def\Prob{\mathop{\mathrm{Prob}}\nolimits}

\def\initial{\mathop{\mathrm{in}}\nolimits}

\providecommand{\abs}[1]{\lvert#1\rvert}
\begin{document}
\title{Gr\"obner bases of contraction ideals}
\author{Takafumi Shibuta
}
\address{Department of Mathematics, Rikkyo University, 
Nishi-Ikebukuro, Tokyo 171-8501, Japan}
\email{shibuta@rikkyo.ac.jp}
\date{}
\baselineskip 15pt
\footskip = 32pt
\begin{abstract}
We investigate Gr\"obner bases of contraction ideals under some monomial homomorphisms. 
As an application of our theorem, we generalize the result of Aoki--Hibi--Ohsugi--Takemura and Hibi-Ohsugi.
Using our results, one can provide many examples of toric ideals that admit square-free or quadratic initial ideals. 
\end{abstract}
\maketitle
\tableofcontents
\section{Introduction}
We denote by $\N=\{0, 1, 2, 3, \dots\}$ the set of non-negative integers. 
For a given positive integer $n$, $[n] = \{1, 2, \dots, n\}$ denotes the set of the first $n$ positive integers. 
For a multi-index $\a={}^t(a_1, \dots, a_r)$ and variables $\x=(x_1, \dots, x_r)$, 
we write $\x^\a={x}_1^{a_1}\cdots {x}_r^{a_r}$. 
We set $\abs{\a}=a_1+\dots+a_r$. 
In this paper, ``quadratic" means ``of degree at most two". 
We say that a monomial ideal $J$ satisfies a property $P$ (e.g. quadratic, square-free, or of degree $m$) 
if the minimal system of monomial generators of $J$ satisfies $P$. 

In recent years, applications of commutative algebras in statics have been successfully developed 
since the pioneering work of Diaconis--Sturmfels \cite{DS}. 
They gave algebraic algorithms for sampling from a finite sample space using Markov chain Monte Carlo methods. 
In some statical models, sample spaces are described as $\A$-fiber space 
$\Fiber_\A(\b)\cap \N^n=\{\a={}^t(a_1, \dots, a_n)\in \N^n \mid \A\cdot \a =\b\}$ of $\b$ 
for some $m\times n$ integer matrix $\A$ and $\b\in \N^n$. 
The {\em toric ideal} of the matrix $\A$ is the binomial prime ideal 
$P_\A=\langle \x^\a-\x^\b\mid \a, \b\in \N^n, \A\cdot\a=\A\cdot\b\rangle$ 
in the polynomial ring $K[\x]=K[x_1, \dots, x_n]$. 
In case where $P_\A$ is a homogeneous ideal, 
Diaconis--Sturmfels shows that using a system of generators of $P_\A$, 
one can construct a connected Markov chain over the finite set $\Fiber_\A(\b)$ for any $\b$. 

We are interested in when $P_\A$ admits a quadratic initial ideal or a square-free initial ideal. 
In this paper, we consider the case where $\A$ is $\B\cdot \C$, the product of two matrices $\B$ and $\C$. 
Defining ideal of Veronese subrings of toric algebras, Segre products of toric ideals, and 
toric fiber products of toric ideals are examples of toric ideals of form $I_{\A}$ with $A=\B\cdot \C$ for some $\B$ and $\C$. 
We will show that in case where $\C$ has a good symmetric structure related with $\B$, 
$P_\A$ admits a quadratic (resp. square-free) initial ideal 
if both of $P_\B$ and $P_\C$ admit quadratic (resp. square-free) initial ideals. 
As an application of our result, we generalize the result of \cite{AHOT} and \cite{HO} about nested configurations. 

Let $R=K[{x}_1, \dots, {x}_r]$ and $S=K[{y}_1, \dots, {y}_s]$ be polynomial rings over a field $K$, 
$I$ an ideal of $S$, and $\phi :R\to S$ a ring homomorphism. 
We call the ideal $\phi^{-1}(I)$ the {\it contraction ideal } of $I$ under $\phi$. 
In this paper, we will investigate when the contraction ideal $\phi^{-1}(I)$ admits a square-free initial ideal, 
or a quadratic initial ideal. 
The one of the most important ring homomorphisms in combinatorics and algebraic statistics are monomial homomorphisms: 
Let $\A=(\a^{(1)}, \dots, \a^{(s)})$ be a $\mu\times s$ integer matrix with column vectors 
$\a^{(i)}={}^t(a^{(i)}_{1}, \dots, a^{(i)}_{\mu})\in \Z^\mu$, and consider the ring homomorphism 
\begin{eqnarray*}
\phi_\A:~ S=K[{y}_1, \dots, {y}_s]&\to& K[{z}_1^{\pm 1}, \dots, {z}_\mu^{\pm 1}]\\
{y}_i &\mapsto& \z^{\a^{(i)}}=\prod_{j=1}^{d}{z}_j^{a^{(i)}_j}. 
\end{eqnarray*}
We call $\phi_\A$ a {\it monomial homomorphism}. 
Using abusive notation, we sometime confound the matrix $\A=(\a^{(1)}, \dots, \a^{(s)})$ with the set $\{\a^{(1)}, \dots, \a^{(s)}\}$ 
if $\a^{(i)}\neq\a^{(j)}$ for all $i\neq j$. 
We denote by $\N\A=\{\sum_{i} n_i\a^{(i)}\mid n_i\in \N\}$ the affine semigroup generated by column vectors of $\A$, 
and by $K[\A]$ the monomial $K$-algebra $K[\z^{\a^{(1)}}, \dots, \z^{\a^{(r)}}]$. 
Note that $P_\A = \Ker{\phi_\A}$. 
We call $\A$ a {\it configuration} if there exists a vector 
$0\neq\lambda=(\lambda_1, \dots, \lambda_\mu)\in \Q^\mu$ such that $\lambda\cdot \a^{(i)}=1$ for all $i$. 
If $\A$ is a configuration, then $P_\A$ is a homogeneous ideal in the usual sense and 
some algebraic properties of $K[\A]\cong S/P_\A$ can be derived from Gr\"obner bases of $P_\A$: 
If $\initial_\prec(P_\A)$ is generated by square-free monomials, then $K[\A]$ is normal. 
For a homogeneous ideal $I\subset S$, if $I$ has a quadratic initial ideal with respect to some term order, 
then $S/I$ is a Koszul algebra, that is, the residue field $K$ has a linear minimal graded free resolution. 
In this paper, we prove that in some cases, if both of $I$ and $\Ker \phi_\A$ admit square-free or quadratic initial ideals, 
then so does the contraction ideal $\phi_\A^{-1}(I)$ (Theorem \ref{main}). 
As a corollary, we obtain the next Theorem. 
\begin{theorem0}[Theorem \ref{main}]\label{thm1}
Let $\V=(\vv_1,\dots, \vv_s)$, $\vv_i\in \Z^d$, be a $d\times s$ integer matrix. 
Let $S=K[y_1,\dots,y_s]$ be a $\Z^d$-graded polynomial ring with $\deg_{\Z^d}(y_i)=\vv_i$, and $\H\subset \Z^d$ a finitely generated subsemigroup. 
Let $\A$ be a system of generators as a semigroup of 
\[
\{\a={}^t(a_1,\dots,a_s)\in \N^s\mid \V\cdot \a \in \H\}. 
\]
Suppose tha $\A$ is a finite set. 
Let $I\subset S$ be a $\Z^d$-graded ideal. 
Then the following hold. 
\begin{enumerate}
\item If both of $I$ and $P_\A$ admit initial ideals of degree at most $m$, 
then so is $\phi_{{\A}}^{-1}(I)$. 
\item If both of $I$ and $P_\A$ admit square-free initial ideals, 
then so is $\phi_{{\A}}^{-1}(I)$. 
\end{enumerate}
\end{theorem0}

Our interest is primarily in case where $I$ is a toric ideal. 
In this case, $\phi_{\A}^{-1}(I)$ is also a toric ideal. 
Let $\A=(\a^{(1)}, \dots, \a^{(r)})$ be an $s\times r$ matrix with $\a^{(i)}\in \N^s$, 
and $\B=(\b^{(1)}, \dots, \b^{(s)})$ a $\mu\times s$ matrix with $\b^{(j)}\in \Z^\mu$. 
Then we obtain monomial homomorphisms 
\[
R\xrightarrow{\phi_\A} S\xrightarrow{\phi_\B}K[{z}_1^{\pm 1}, \dots, {z}_\mu^{\pm 1}]. 
\]
Note that the composition of monomial homomorphism is also monomial homomorphism 
defined by the product of the matrices, $\phi_{\B}\circ \phi_{\A}=\phi_{\B\cdot \A}$, 
and the toric ideal $P_{\B\cdot\A}$ is the contraction ideal $\phi_\A^{-1}(P_\B)$. 
By Theorem \ref{thm1}, the following theorem holds. 
\begin{theorem0}[Theorem \ref{generalized nested}]\label{thm2}
Let $d>0$ and $\lambda_i\in \N$ for $i\in [d]$ be integers, $\Z^d=\bigoplus_{i=1}^d \Z\e_i$ a free $\Z$-module of rank $d$ 
with a basis $\e_1, \dots, \e_d$, and 
$S=K\bigl[y^{(i)}_j\mid i\in [d], ~j\in[\lambda_i]\bigr]$ a $\Z^d$-graded polynomial ring with $\deg y^{(i)}_j=\e_i$ for $i \in [d]$, $j\in[\lambda_i]$. 
Let $\A\subset\N^d= \bigoplus_{i=1}^d\N \e_i$ be a configuration. We set 
\[
\tilde{\A}=\Bigl\{\a=(a^{(i)}_j\mid i\in [d], ~j\in[\lambda_i]\bigr)\in \bigoplus_{i=1}^d\Z^{\lambda_i} \mid \deg_{\Z^d} \y^\a\in \A\Bigr\}, 
\]
$R=K[x_\a\mid \a\in \tilde{\A}]$, and $\phi_{\tilde{\A}}:R\to S$, $x_\a\mapsto \y^\a$. 
Let $\V$ be a $\nu\times \mu$ integer matrix, and $\w_1,\dots,\w_d\in \Z^\nu$ linearly independent vectors. 
For $i\in[d]$, we fix a finite set 
\[
\B_i = \bigl\{\b^{(i)}_j\mid j\in[\lambda _i]\}\subset \Fiber_\V(\w_i)=\{\b \in \Z^\mu\mid \V\cdot\b=\w_i\bigr\}\subset\Z^\mu. 
\]
We set $\B=\B_1\cup\dots\cup\B_d$ which is a configuration of $\Z^\mu$. 
We set 
\[
\A[\B_1,\dots,\B_d]:=\biggl\{\sum_{i\in[d], j\in[\lambda_i]}a^{(i)}_j\b^{(i)}_j\mid \a=(a^{(i)}_j\mid i\in[d], j\in[\lambda_i])\in \tilde{\A}\biggr\}. 
\]
Then the following hold. 
\begin{enumerate}
\item If both of $P_\B$ and $P_\A$ admit initial ideals of degree at most $m$, 
then so is $P_{\A[\B_1,\dots,\B_d]}$. 
\item If both of $P_\B$ and $P_\A$ admit square-free initial ideals, 
then so is $P_{\A[\B_1,\dots,\B_d]}$. 
\end{enumerate}
\end{theorem0}
Note that $\tilde{\A}$ is a nested configuration defined in \cite{AHOT}. 
This theorem is a generalization of the theorem of Aoki--Hibi--Ohsugi--Takemura \cite{AHOT} and Hibi-Ohsugi \cite{HO}, 
and contains the result of Sullivant \cite{Sullivant} (toric fiber products). 
The easiest case is the case where there exists $\mu_1, \dots, \mu_d\in\N$ such that $\mu=\mu_1+\dots+\mu_d$, $\N^\mu=\N^{\mu_1}\times\dots\times\N^{\mu_d}$, 
and $\B_i\subset \{\bold{0}\}\times\dots \times\{\bold{0}\}\times\N^{\mu_i}\times\{\bold{0}\}\times\dots \times\{\bold{0}\}$ for all $i$. 
If this is the case, the Gr\"obner basis of $P_\B$ is the union of Gr\"obner bases of $P_{\B_i}$'s, 
and Theorem \ref{thm2} is follows from Hibi-Ohsugi \cite{HO} Theorem 2.6. 
It often happens that the product $\B\cdot \A$ of two matrices $\B$ and $\A$ has two equal columns. 
If $\A=(\a^{(1)}, \dots, \a^{(r)})$ is a matrix with $\a^{(i)}=\a^{(j)}$, 
and $\prec$ is a term order on $R$ such that ${x}_i\prec {x}_j$. 
Let $\A'=(\a^{(1)}, \dots, \a^{(i-1)}, \a^{(i+1)}, \dots, \a^{(r)})$. 
Then the union of $\{{x}_j-{x}_i\}$ and a Gr\"obner basis of $P_{\A'}\subset K[{x}_1, \dots, {x}_{j-1}, {x}_{j+1}, \dots, {x}_{r}]$ 
with respect to the term order induced by $\prec$ is a Gr\"obner basis of $P_\A$. 
Therefore a Gr\"obner basis of $P_{\A}$ and that of $P_{\A'}$ are essentially equivalent. 
In particular, the maximal degree of the minimal system of monomial generators of $\initial_\prec(P_\A)$ coincides with that of $\initial_\prec(P_{\A'})$, 
and $\initial_\prec(P_\A))$ is generated by square-free monomials if and only if $\initial_\prec(P_{\A'})$ is. 

We will present a small example. 
Assume that there are four ingredients $z_1, z_3, z_3, z_4$, and three manufacturers $B_1, B_2, B_3$. 
Assume that each ingredient $z_i$ is equipped with a property vector $\vv_i=(v_{i1}, v_{i2}, v_{i3})\in \N^3$ as in Table \ref{Ingredient}. 
\begin{table}[!h]
  \begin{center}
		\begin{tabular}{|c||c|c|c|}
			\hline
			Ingredient & Property 1 & Property 2 & Property 3 \\
			\hline\hline
			$z_1$ & 600 & 30 & 20 \\
			$z_2$ & 400 & 30 & 10 \\
			$z_3$ & 700 & 20 & 30 \\
			$z_4$ & 1200 & 40 & 50 \\
			\hline
			\end{tabular}
  \end{center}
\caption{Ingredient}
\label{Ingredient}
\end{table}
For example, $z_i$'s are cooking ingredients and $\vv_i$ is the list of the nutritions of $z_i$, 
or $z_i$'s are financial products and $\vv_i$ is the list of the price, the interest income, and the risk index of $z_i$. 
Each of manufacturers provides products combining $z_1, \dots, z_4$. 
A product is expressed as a monomial $z_1^{b_1}z_2^{b_2}z_3^{b_3}z_4^{b_4}$ where $b_i$ is the number of $z_i$ contained in the product. 
Assume that property vectors are additive, that is, the property vector of $z_1^{b_1}z_2^{b_2}z_3^{b_3}z_4^{b_4}$ is $b_1\vv_1+\dots+b_4\vv_4$. 
Suppose that each manufacturer $B_j$ sells the products with a fixed property vector $\w_j$ as in Table \ref{Products}. 
\begin{table}[!h]
  \begin{center}
		\begin{tabular}{|c||c|c|}
		\hline
		Manufacturer & products & property vector $\w_j$\\
		\hline\hline
		$B_1$ & $\bar{y}_{1}:=z_1^2z_3z_4$, \hspace{2mm} $\bar{y}_{2}:=z_1z_2z_3^3$ & (3100, 120, 120)\\
		\hline
		$B_2$ & $\bar{y}_{3}:=z_1z_3z_4^2$, \hspace{2mm} $\bar{y}_{4}:=z_2z_3^3z_4$ & (3700, 130, 150)\\
		\hline
		$B_3$ & $\bar{y}_{5}:=z_1^2z_4^2$, $\bar{y}_{6}:=z_1z_2z_3^2z_4$, $\bar{y}_{7}:=z_2^2z_3^4$ & (3600, 140, 140)\\
		\hline
		\end{tabular}
	\end{center}
		\caption{Products}
		\label{Products}
\end{table}
Suppose that each customer choose two manufacturers and buys one product from each chosen manufacturer. 
Then there are $16$ patterns of choice as 
$\bar{y}_{1}\bar{y}_{3}, \bar{y}_{1}\bar{y}_{4}, \bar{y}_{2}\bar{y}_{3}, \bar{y}_{2}\bar{y}_{4}$, 
$\bar{y}_{1}\bar{y}_{5}, \bar{y}_{1}\bar{y}_{6}, \bar{y}_1\bar{y}_7, \bar{y}_{2}\bar{y}_{5}, \bar{y}_{2}\bar{y}_{6}, \bar{y}_2\bar{y}_7$, ~~ 
$\bar{y}_{3}\bar{y}_{5}, \bar{y}_{3}\bar{y}_{6}, \bar{y}_3\bar{y}_7, \bar{y}_{4}\bar{y}_{5}, \bar{y}_{4}\bar{y}_{6}, \bar{y}_4\bar{y}_7$, 
and we name them $\bar{x}_1, \dots, \bar{x}_{16}$, respectively. 
Suppose that there are $1000$ customers, and the choices of the customers is 
\[
\c_0={}^t(101, 59, 80, 21, 129, 62, 78, 83, 47, 51, 98, 70, 12, 58, 31, 20)
\]
where the $k$-th component of $\c_0$ is the number of customers whose choice is $\bar{x}_k$. 
We will count the number of $\bar{y}_j$ that are sold, and the number of $z_i$ in the whole of the sold products. 
Let 
\makeatletter
\c@MaxMatrixCols=16
\makeatother
\begin{equation*}
\tilde{\A}=
{\scriptsize
\begin{pmatrix}
1&1&0&0& 1&1&1&0&0&0& 0&0&0&0&0&0\\
0&0&1&1& 0&0&0&1&1&1& 0&0&0&0&0&0\\
1&0&1&0& 0&0&0&0&0&0& 1&1&1&0&0&0\\
0&1&0&1& 0&0&0&0&0&0& 0&0&0&1&1&1\\
0&0&0&0& 1&0&0&1&0&0& 1&0&0&1&0&0\\
0&0&0&0& 0&1&0&0&1&0& 0&1&0&0&1&0\\
0&0&0&0& 0&0&1&0&0&1& 0&0&1&0&0&1
\end{pmatrix}
}, \hspace{2mm}
\B=
{\scriptsize
\begin{pmatrix}
2&1& 1&0& 2&1&0\\
0&1& 0&1& 0&1&2\\
1&3& 1&3& 0&2&4\\
1&0& 2&1& 2&1&0
\end{pmatrix}
}. 
\end{equation*}
Then the $j$-th component of $\tilde{\A}\cdot\c_0={}^t(429, 282, 361, 189, 368, 210, 161)$ is the number of $\bar{y}_j$ that are sold, 
and the $i$-th component of $\B\cdot(\tilde{\A}\cdot\c_0)={}^t(2447, 1003, 3267, 2286)$ is the number of $z_i$ in the whole of the sold products. 
We consider all the possibilities of $1000$ customers choices 
such that the number of $z_i$ in the whole of the sold products is the same as $\c_0$ for all $i=1, 2, 3, 4$. 
This space is expressed as the $(\B\cdot \tilde{\A})$-fiber space of $\b:=(\B\cdot\tilde{\A})\cdot\c_0={}^t(2447, 1003, 3267, 2286)$; 
\[
\Fiber_{\B\cdot \tilde{\A}}(\b)\cap\N^{16}=\{\c\in \N^{16}\mid (\B\cdot\tilde{\A})\cdot\c=\b\}. 
\]
Gr\"obner bases of the toric ideal of 
\begin{equation*}
\B\cdot\tilde{\A}=
{\scriptsize
\begin{pmatrix}
3&2&2&1&4&3&2&3&2&1&3&2&1&2&1&0\\
0&1&1&2&0&1&2&1&2&3&0&1&2&1&2&3\\
2&4&4&6&1&3&5&3&5&7&1&3&5&3&5&7\\
3&2&2&1&3&2&1&2&1&0&4&3&2&3&2&1
\end{pmatrix}
}.
\end{equation*}
is used to analyze this model. 
The toric ideal $P_{\B\cdot\tilde{\A}}$ is complicated, and generated by $33$ binomials 
\begin{eqnarray*}
\hspace{-7mm}&&{\scriptstyle 
x_{16}x_{14}-x_{15}^2,~ x_{13}-x_{15},~ x_{12}-x_{14},~ x_{16}x_{11}-x_{15}x_{14},~ x_{15}x_{11}-x_{14}^2,~ x_{16}x_{9}-x_{15}x_{10},~ 
x_{15}x_{9}-x_{14}x_{10},~ x_{14}x_{9}-x_{11}x_{10}},~ \\
\hspace{-7mm}&&{\scriptstyle 
x_{16}x_{8}-x_{14}x_{10},~ x_{15}x_{8}-x_{11}x_{10},~ x_{14}x_{8}-x_{11}x_{9},~ x_{10}x_{8}-x_{9}^2,~ x_{7}-x_{9},~ x_{6}-x_{8},~ 
x_{16}x_{5}-x_{11}x_{10},~ x_{15}x_{5}-x_{11}x_{9},~ }\\
\hspace{-7mm}&&{\scriptstyle 
x_{14}x_{5}-x_{11}x_{8},~ x_{10}x_{5}-x_{9}x_{8},~ x_{9}x_{5}-x_{8}^2,~ x_{16}x_{3}-x_{15}x_{4},~ x_{15}x_{3}-x_{14}x_{4},~ x_{14}x_{3}-x_{11}x_{4},~ 
x_{10}x_{3}-x_{9}x_{4},~ x_{9}x_{3}-x_{8}x_{4}},~ \\
\hspace{-7mm}&&{\scriptstyle 
x_{8}x_{3}-x_{5}x_{4},~ x_{2}-x_{3},~ x_{16}x_{1}-x_{14}x_{4},~ x_{15}x_{1}-x_{11}x_{4},~ 
x_{14}x_{1}-x_{11}x_{3},~ x_{10}x_{1}-x_{8}x_{4},~ x_{9}x_{1}-x_{5}x_{4},~ x_{8}x_{1}-x_{5}x_{3},~ x_{4}x_{1}-x_{3}^2
}.
\end{eqnarray*}
On the other hand, $\B$ has a simple the toric ideal, and $\tilde{\A}$ has a good combinatorial structure. 
The Gr\"obner basis of $P_\B$ with respect to the lexicographic order $\prec_{\rm lex}$ with $y_7\prec_{\rm lex}\dots\prec_{\rm lex} y_1$ is 
\[
\{y_4y_1-y_3y_2, y_6y_1-y_5y_2, y_7y_1-y_6y_2, y_6y_3-y_5y_4, y_7y_3-y_6y_4, y_7y_5-y_6^2 \}, 
\]
thus $\initial_{\prec_{\rm lex}}(P_\B)$ is generated by square-free quadratic monomials, 
and $P_{\tilde{\A}}$ admits a square-free quadratic initial ideal since $\tilde{\A}$ is a nested configuration (\cite{AHOT} Theorem 3.6). 
By Theorem \ref{thm2}, 
we conclude that $P_{\B\cdot\tilde{\A}}$ also admits a square-free quadratic initial ideal. 
\section{Preliminaries on Gr\"obner bases}
Here, we recall the theory of Gr\"obner bases. 
See \cite{Cox1}, \cite{Cox2} and \cite{Sturmfels} for details. 

Let $R=K[{x}_1, \dots, {x}_r]$ be a polynomial ring over a field $K$. 
A total order $\prec$ on the set of monomials $\{\x^\a\mid \a\in\N^r\}$ is 
a {\em term order} on $R$ if $\x^0=1$ is the unique minimal element, and $\x^{\a}\prec \x^{\b}$ implies 
$\x^{\a +\c }\prec \x^{\b+\c}$ for all $\a$, $\b$, $\c\in \N^r$. 
Let $\a ={}^t( a_1, \ldots, a_r)$ and 
$\b ={}^t( b_1, \ldots, b_r)\in \N^r$. 
\begin{Definition}[lexicographic order] 
The term order $\prec_{lex}$ called a {\em lexicographic order} with $x_r\prec_{lex} \cdots\prec_{lex} x_1$ is defined as follows: 
$\x^{\a }\prec_{lex} \x^{\b }$ if 
$a _j{<}b _j~\mbox{where}~j=\min\{i\mid a _i\neq b _i\}$. 
\end{Definition}
\begin{Definition}[reverse lexicographic order] 
The term order $\prec_{rlex}$ called a 
{\em reverse lexicographic order} with $x_r\prec_{rlex} \cdots\prec_{rlex} x_1$ is defined as follows: 
$\x^{\a }\prec_{rlex} \x^{\b }$ if 
$\abs{\a}<\abs{\b}$ or $\abs{\a}=\abs{\b}$ and 
$a _j{>}b _j~\mbox{where}~j=\min\{i\mid a _i\neq b _i\}$. 
\end{Definition}
\begin{Definition}
Let $\prec$ be a term order on $R$, $f\in R$, and $I$ an ideal of $R$. 
The {\it initial term} of $f$, denoted by $\initial_\prec(f)$, is the highest term of $f$ with respect to $\prec$. 
We call $ \initial_\prec(I)=\langle \initial_\prec(f)\mid f\in I\rangle$ the {\it initial ideal} of $I$ with respect to $\prec$. 
We say that a finite collection of polynomials $G\subset I$ is a {\it Gr\"obner basis} of $I$ 
with respect to $\prec$ if $\langle \initial_\prec(g)\mid g\in G\rangle = \initial_\prec(I)$. 
\end{Definition}
\begin{Remark}\label{monomial basis}
The set of all monomials not in the initial ideal $\initial_\prec(I)$ forms a basis of $R/I$ as a $K$-vector space. 
\end{Remark}
We recall the Buchberger's Criterion. 
We denote by $\overline{f}^G$ a remainder of $f$ on division by $G$. 
If $G$ is a Gr\"obner basis of $I$ with respect to $\prec$,  $\overline{f}^G$ is the unique polynomial $g$ 
such that $f-g\in I$ and any term of $g$ is not in $\initial_\prec(I)$. 
\begin{Definition}
The {\it S-polynomial} of $f$ and $g$ is given by 
\[
S(f, g)= \frac{\lcm(\initial_\prec(f), \initial_\prec(g))}{\initial_\prec(f)}f 
- \frac{\lcm(\initial_\prec(f), \initial_\prec(g))}{\initial_\prec(g)}g. 
\]
\end{Definition}
\begin{Theorem}[Buchberger's Criterion] Let $I \subset K[{x}_1, \dots, {x}_r]$ be an ideal with a system of generators $G$. 
Then $G$ is a Gr\"obner basis for $I$ if and only if $\overline{S(f, g)}^G=0$ for all $f, g\in G$. 
\end{Theorem}
We can also consider initial ideals with respect to weight vectors. 
Fix a weight vector $\w=(w_1, \dots, w_r)\in\N^r$. 
We grade the ring $R$ by associating weights $w_i$ to ${x}_i$. 
Then the weight of $\x^\a\in R$ is $\w\cdot \a$. 

\begin{Definition}
Given a polynomial $f\in R$ and a weight vector $\w$, the {\it initial form} $\initial_\w(f)$ is the sum of all monomials of $f$ 
of the highest weight with respect to $\w$. 
We call $\initial_\w(I)=\langle \initial_\w(f)\mid f\in I\rangle$ the {\it initial ideal} of $I$ with respect to $\w$. 
We say that a finite collection of polynomials $G\subset I$ is a {\it pseudo-Gr\"obner basis} of $I$ 
with respect to $\w$ if $\langle \initial_\w(g)\mid g\in G\rangle = \initial_\w(I)$. 
If $G$ is a pseudo-Gr\"obner basis and $\initial_\w(g)$ is a monomial for all $g\in G$, 
we call $G$ a Gr\"obner basis of $I$ with respect to $\w$. 
\end{Definition}
It is easy to prove that any pseudo-Gr\"obner basis of an ideal is a system of generators of the ideal. 

We define a new term order $\prec_\w$ defined by weight vector $\w$ with a term order $\prec$ as a tie-breaker. 
\begin{Definition}
For a weight vector $\w$ and a term order $\prec$, we define a new term order $\prec_\w$ as follows: 
$\x^\a \prec_\w \x^\b $ if $\w\cdot\a <\w\cdot\b $, 
or $\w\cdot\a =\w\cdot\b $ and $\x^\a \prec \x^\b $. 
\end{Definition}
A Gr\"obner basis of $I$ with respect to $\prec_\w$ is a pseudo-Gr\"obner basis of $I$ with respect to $\w$, 
but the converse is not true in general. 
We end this section with the following useful and well-known propositions. 
See \cite{Sturmfels} for the proofs. 
\begin{Proposition}
For any term order $\prec$ and any ideal $I\subset R$, 
there exists a vector $\w\in\N^r$ such that $\initial_\prec(I)= \initial_\w(I)$. 
\end{Proposition}
\begin{Proposition}\label{w}
$\initial_\prec(\initial_\w(I))=\initial_{\prec_\w}(I)$. 
\end{Proposition}
\section{Gr\"obner bases of contraction ideals}
Rings appearing in this paper may equip two or more graded ring structures. 
To avoid the confusion, for a ring with a graded ring structure given by an object $*$, 
(for example, a weight vector $\w$, an abelian group $\Z^d$, and so on) 
we say that elements or ideals are $*$-homogeneous or $*$-graded if they are homogeneous 
with respect to the graded ring structure given by $*$. 
For a $*$-homogeneous polynomial $f$, 
we denote by $\deg_*(f)$ the degree of $f$ with respect to $*$. 

Let $R=K[{x}_1, \dots, {x}_r]$ and $S=K[{y}_1, \dots, {y}_s]$ be polynomial rings over $K$, 
and $I$ an ideal of $S$. 
Let $\A=\{\a^{(1)}, \dots, \a^{(r)}\}\subset \N^s$, 
and $\phi_\A:R\to S$ $({x}_j\mapsto \y^{\a^{(j)}})$ be a monomial homomorphism. 
Fix a term order $\prec$ on $R$. 
We investigate when the contraction of an ideal $\phi_\A^{-1}(I)$ has square-free or quadratic initial ideals. 
It is expected that if both of $I$ and $P_\A=\Ker \phi_\A$ have such properties, then so does $\phi_\A^{-1}(I)$. 
We will prove this holds true under some assumptions. 
First, we prove this in case where $I$ is a monomial ideal, 
then reduce the general cases to monomial ideal cases. 
\subsection{In case of monomial ideals}
In this subsection, we consider contractions of monomial ideals. 
\begin{Definition}
For a monomial ideal $J$, 
we denote by $\delta (J)$ the maximum of the degrees of a system of minimal generators of $J$. 
\end{Definition}
\begin{Definition}
Let $I\subset S$ be a monomial ideal. 
Let $L^{(\A)}_\prec(I)$ be a monomial ideal generated by all monomials in $\phi_\A^{-1}(I)\backslash \initial_\prec(P_\A)$. 
We denote by $M^{(\A)}_\prec(I)$ the minimal system of monomial generators $L^{(\A)}_\prec(I)$ 
\end{Definition}
Let $I_1, I_2\subset S$ be monomial ideals. 
Then 
\[
L^{(\A)}_\prec(I_1+I_2) =L^{(\A)}_\prec(I_1) +L^{(\A)}_\prec(I_2) 
\]
as $\phi_\A(u)\in I_1+I_2$ if and only if $\phi_\A(u)\in I_1$ or $\phi_\A(u)\in I_2$ for a monomial $u$. 
In particular, if $I$ is generated by monomials $\y^{\b_1},\dots,\y^{\b_n}$, 
then 
\[
L^{(\A)}_\prec(I)=L^{(\A)}_\prec(\y^{\b_1})+\dots+L^{(\A)}_\prec(\y^{\b_n}). 
\]
\begin{Lemma}\label{pull-back of monomial ideals}
Let $I\subset S$ be a monomial ideal. 
Then the following hold: 
\begin{enumerate}
\item $\delta(L^{(\A)}_\prec(I))\le\delta (I)$. 
\item If $I$ is generated by square-free monomials, then $L^{(\A)}_\prec(I)$ is generated by square-free monomials. 
\end{enumerate}
\end{Lemma}
\begin{proof}
It is enough to prove in case where $I$ is generated by one monomial $v$.  
Let ${u} \in \phi_\A^{-1}(I)\backslash \initial_\prec(P_\A)$ be a monomial. 

$(1)$ 
Let $\delta:=\delta (I)=\deg(v)$ and $m=\deg(u)$. 
Assume that $m>\delta$. 
It is enough to show that there exists a monomial ${u'}\in \phi_\A^{-1}(I)$ 
of degree strictly less than $m$ such that ${u'}$ divides ${u} $. 
We prove this by induction on $\delta$. 
It is trivial in case where $\delta=0$. Assume that $\delta\ge 1$. 
We may assume, without loss of generality, that ${x}_1$ divides $u$. 
Let $\tilde{v}=\gcd({v} , \phi_\A({x}_1))$. 
If $\tilde{v} =1$, we can take ${u} /{x}_1$ as ${u'}$. 
If $\tilde{v} \neq 1$, then ${v} /\tilde{v} $ is a monomial of degree at most $\delta-1$, 
and ${u} /{x}_1\in \phi_\A^{-1}(\langle {v} /\tilde{v} \rangle)$. 
By the hypothesis of induction, 
there exists a monomial ${u''}\in\phi_\A^{-1}({v} /\tilde{v})$ such that ${u''}$ divides ${u} /{x}_1$ and $\deg({u''})<m-1$. 
Then ${u'}={x}_1\cdot {u''}$ is a monomial with desired conditions. 

$(2)$ We may assume, without loss of generality, that $v=\prod_{j=1}^t{y}_j$ for some $t\le s$. 
Let $\x^{\a}=\prod_{i=1}^{r} {x}_i^{a_i}\in \phi_\A^{-1}(I)$ be a monomial. 
It is enough to show that there exists a square-free monomial in $\phi_\A^{-1}(I)$ that divides $\x^\a$. 
For $1\le k \le t$, there exists $1\le i(k)\le r$ such that $a_{i(k)}\neq 0$ and ${y}_k$ divides $\phi_\A({x}_{i(k)})$. 
Let $\Lambda =\{i(1), \dots, i(t)\}$. Then $\prod_{i\in\Lambda}{x}_i$ is a square-free monomial in $\phi_\A^{-1}(I)$ which divides $\x^\a$. 
\end{proof}
\begin{Proposition}\label{monomial cases}
Let $I\subset S$ be a monomial ideals. 
Let $G_\A$ be a Gr\"obner basis of $P_\A$ with respect to $\prec$. 
Then the following hold: 
\begin{enumerate}
\item $G_\A\cup M^{(\A)}_\prec(I)$ is a Gr\"obner basis of $\phi_\A^{-1}(I)$ with respect to $\prec$. 
\item $\delta(\initial_\prec(\phi_\A^{-1}(I)))\le \max\{\delta(I), \delta(\initial_\prec(P_\A))\}$
\item If $I$ and $\initial_\prec(P_\A)$ are generated by square-free monomials, 
then $\initial_\prec(\phi_\A^{-1}(I))$ is generated by square-free monomials. 
\end{enumerate}
\end{Proposition}
\begin{proof}
It is clear that $G_\A\cup M^{(\A)}_\prec(I)\subset \phi_\A^{-1}(I)$. 
Let $f\in \phi_\A^{-1}(I)$, and $g$ the remainder of $f$ on division by $G_\A$. 
Then any term of $g$ is not in $\initial_\prec(P_\A)$. 
Hence different monomials appearing in $g$ map to different monomials under $\phi_\A$. 
Since $I$ is a monomial ideal, it follows that all terms of $g$ are in $L^{(\A)}_\prec(I)$. 
Thus the remainder of $g$ on division by $M^{(\A)}_\prec(I)$ is zero. 
Therefore a remainder of $f$ on division by $G\cup M^{(\A)}_\prec(I)$ is zero. 
This implies (1). 

We conclude (2) and (3) immediately from (1) and the Lemma \ref{pull-back of monomial ideals}. 
\end{proof}
\subsection{Reduction to case of monomial ideals}
Let $I$ be an ideal of $S$. 
We fix a weight vector $\w=(w_1, \dots, w_s)\in \N^s$ on $S=K[{y}_1, \dots, {y}_s]$ such that $\initial_\w(I)$ 
is a monomial ideal. 
We take 
\[
\phi_\A^{*}\w:=\w\cdot \A=(\deg_\w\phi_\A({x}_1), \dots, \deg_\w\phi_\A({x}_r)) 
\]
as a weight vector on $R=K[{x}_1, \dots, {x}_r]$ so that $\phi_\A$ preserve weight. 
Then $R$ and $S$ are $\N$-graded rings; $R=\bigoplus_{i\in \N}R_i$ and $S=\bigoplus_{i\in \N}S_i$ where $R_i$ and $S_i$ 
are the $K$-vector space spanned by all monomials of weight $i$ with respect to $\phi_\A^{*}\w$ and $\w$, respectively. 
Then $\phi_\A$ is a homogeneous homomorphism of graded rings of degree $0$, that is, 
$\phi_\A(R_i)\subset S_i$. 
Hence $P_\A$ is a $\phi_\A^{*}\w$-homogeneous ideal. 
It is easy to show that $\initial_{\phi_\A^{*}\w}(\phi_\A^{-1}(I))\subset \phi_\A^{-1}(\initial_\w(I))$ 
(see Lemma \ref{surjection} (1)). 
This inclusion is not always an equality. 
\begin{Example}
Let $R=K[{x}_1, {x}_2]$ and $S=K[{y}_1, {y}_2]$ be polynomial rings, 
$\w=(2, 1)$ a weight vector on $S$, 
and 
$\A=\left(
	\begin{array}{cc}
	1 & 1 \\
	0 & 1 \\
	\end{array}
\right)$.  
Then $\phi_\A({x}_1)={y}_1$, $\phi_\A({x}_2)={y}_1{y}_2$, and $\phi_\A^{*}\w=(2, 3)$. 
Let $I$ be an ideal generated by $f={y}_1+{y}_2\in S$. 
Then 
\[
\phi_\A^{-1}(I)=\langle {x}_1-{y}_1, {x}_2-{y}_1{y}_2, f\rangle\cap R=\langle {x}_1^2+{x}_2\rangle 
\]
and thus $\initial_{\phi_\A^{*}\w}(\phi_\A^{-1}(I))=\langle {x}_1^2\rangle$. 
On the other hand, $\initial_\w(f)={y}_1$ and thus 
\[
\phi_\A^{-1}(\initial_\w(I))=\langle {x}_1, {x}_2\rangle. 
\]
Therefore $\initial_{\phi_\A^{*}\w}(\phi_\A^{-1}(I))\neq \phi_\A^{-1}(\initial_\w(I))$. 
\end{Example}
In the case where the equality 
$\initial_{\phi_\A^{*}\w}(\phi_\A^{-1}(I))= \phi_\A^{-1}(\initial_\w(I))$
holds, we obtain the next theorem. 
\begin{Theorem}\label{if=}
Let the notation be as above. 
Suppose, in addition, that the equality 
\[
\initial_{\phi_\A^{*}\w}(\phi_\A^{-1}(I))= \phi_\A^{-1}(\initial_\w(I)) 
\]
holds. 
Then the following hold: 
\begin{enumerate}
\item 
$\initial_{\prec_{\phi_\A^{*}\w}}(\phi_\A^{-1}(I))=\initial_\prec(P_\A)+L^{(\A)}_\prec(\initial_\w(I))$. 
\item 
$\delta(\initial_{\prec_{\phi_\A^{*}\w}}(\phi_\A^{-1}(I)))\le \max\{\delta(\initial_\w(I)), \delta(\initial_\prec(P_\A))\}$. 
\item If both of $\initial_\w(I)$ and $\initial_\prec(P_\A)$ are generated by square-free monomials, 
then $\initial_{\prec_{\phi_\A^{*}\w}}(\phi_\A^{-1}(I))$ is generated by square-free monomials. 
\end{enumerate}
\end{Theorem}
\begin{proof}
Since 
\[
\initial_{\prec_{\phi_\A^{*}\w}}(\phi_\A^{-1}(I))=\initial_\prec(\initial_{\phi_\A^{*}\w}(\phi_\A^{-1}(I)))
=\initial_\prec(\phi_\A^{-1}(\initial_\w(I))), 
\]
and $\initial_\w(I)$ is a monomial ideal, 
we obtain the assertion applying Proposition \ref{monomial cases} to the monomial ideal $\initial_\w(I)$. 
\end{proof}
In the rest of this paper, 
we investigate when the equality $\initial_{\phi_\A^{*}\w}(\phi_\A^{-1}(I))= \phi_\A^{-1}(\initial_\w(I))$ holds. 
\subsection{Pseudo-Gr\"obner bases}
To investigate when the equality 
$\initial_{\phi_\A^{*}\w}(\phi_\A^{-1}(I))= \phi_\A^{-1}(\initial_\w(I))$ 
holds, we extend the definition of pseudo-Gr\"obner bases to ideals of $\N$-graded rings. 
\begin{Definition}\label{pGB}
Let $A=\bigoplus_{i\in \N} A_i$ be an $\N$-graded ring and $f=\sum_i f_i\in A$ $(f_i\in A_i)$. 
We define $\initial_A(f)=f_d$ where $d=\deg(f)=\max\{i\mid f_i\neq 0\}$. 
For an ideal $I\subset A$, we define 
\[
\initial_A(I)=\langle \initial_A(f)\mid f\in I\rangle\subset A. 
\]
We say that a finite collection of polynomials $G\subset I$ is a pseudo-Gr\"obner basis of $I$ 
if $\langle \initial_A(g)\mid g\in G\rangle = \initial_A(I)$. 
\end{Definition}
It is easy to show that a pseudo-Gr\"obner basis of $I$ generates $I$. 
Let $A=\bigoplus_{i\in \N} A_i$ and $B=\bigoplus_{i\in \N} B_i$ be graded rings, 
and $\phi:A\to B$ a graded ring homomorphism of degree $0$, that is, $\phi(A_i)\subset B_i$ for all $i$. 
\begin{Lemma}\label{surjection}
Let $I$ be an ideal of $B$. Then the following hold: 
\begin{enumerate}
\item $\initial_A(\phi^{-1}(I))\subset \phi^{-1}(\initial_B(I))$. 
\item If $\phi$ is surjective, then $\initial_A(\phi^{-1}(I))= \phi^{-1}(\initial_B(I))$. 
\end{enumerate}
\end{Lemma}
\begin{proof}
$(1)$ Let $f=\sum_{i=1}^d f_i\in \phi^{-1}(I)$ where $f_i\in A_i$ and $f_d\neq 0$. 
Then $\initial_A(f)=f_d$, $\phi(f)=\sum_{i=1}^d \phi(f_i)\in I$, and $\phi(f_i)\in B_i$. 
Hence $\phi(f_d)=0$ or $\phi(f_d)=\initial_B(\phi(f))\in \initial_B(I)$, and thus $\phi(\initial_A(f))\in \initial_B(I)$. 

$(2)$ Since $A/\Ker\phi\cong B$ as $\N$-graded rings, and $\phi^{-1}(I)/\Ker\phi\cong I$ as $\N$-graded ideals, 
$\initial_A(\phi^{-1}(I))$ coincides with $\phi^{-1}(\initial_B(I))$ module $\Ker \phi$. 
Since $\Ker \phi$ is a homogeneous ideal of $A$, $\Ker\phi\subset\initial_A(\phi^{-1}(I))$, 
and it is clear that $\Ker\phi\subset \phi^{-1}(\initial_B(I))$. 
Hence we conclude the assertion. 
\end{proof}
\subsection{Sufficient condition for that initial commutes with contraction}
Now, we return to the problem 
when the equality $\initial_{\phi_\A^{*}\w}(\phi_\A^{-1}(I))= \phi_\A^{-1}(\initial_\w(I))$ holds. 
The homomorphism $\phi_\A:R\to S$ can be decomposed to the surjection $R\to K[\A]$ and the inclusion $K[\A]\hookrightarrow S$. 
Note that since $K[\A]$ has an $\N$-graded ring structure induced by $\w$, 
we can consider pseudo-Gr\"obner bases of ideals of $K[\A]$ in the sense of Definition \ref{pGB}. 
Note that $\initial_{K[\A]}(f)=\initial_\w(f)$ for $f\in K[\A]$. 
By Lemma \ref{surjection}, the equality 
\[
\initial_{\phi_\A^{*}\w}(\phi_\A^{-1}(I))= \phi_\A^{-1}(\initial_\w(I)) 
\]
holds if and only if the equality 
\[
\initial_{K[\A]}(I\cap K[\A])=\initial_\w(I)\cap K[\A] 
\]
holds. 
To obtain a class of monomial algebra $K[\A]$ such that this equality holds, 
we define subrings of graded rings. 
\begin{Definition}
Let $\G$ be a semigroup, and $S=\bigoplus_{\vv\in \G}S_\vv$ a $\G$-graded ring. 
For a subsemigroup $\H\subset \G$, we define 
\[
S^{(\H)}=\bigoplus_{\vv\in \H}S_\vv, 
\]
a graded subring of $S$. 
\end{Definition}
Let $S=K[{y}_1, \dots, {y}_s]$ a $\Z^d$-graded polynomial ring with $\deg_{\Z^d}({y}_i)=\vv_i$ where $\vv_i\in \Z^d$ is a column vector. 
We set $\V=(\vv_1,\dots,\vv_s)$, and let $\H$ be a finitely generated subsemigroup of $\N\V\subset \Z^d$. 
Then 
\[
S^{(\H)}=K[\y^\a\mid \a\in\N^s, \V\cdot \a \in \H], 
\]
and thus $S^{(\H)}$ is Noetherian if and only if $\{\a\in\N^s \mid \V\cdot \a \in \H\}$ is finitely generated as a semigroup. 
We will give sufficient conditions that $S^{(\H)}$ is Noetherian. 
\begin{Definition}
We say that a semigroup $\H\subset\Z^d$ is {\it normal } if $\H=L\cap C$ 
for some sublattice $L\subset\Z^d$ and finitely generated rational cone $C\subset \Q^d$. 
\end{Definition}
It is well-known that normal semigroups are finitely generated (Gordan's Lemma). 
\begin{Proposition}\label{noetherian condition}
Let $S=K[{y}_1, \dots, {y}_s]$ a $\Z^d$-graded polynomial ring with $\deg_{\Z^d}({y}_i)=\vv_i\in \Z^d$. 
We set $\V=(\vv_1,\dots,\vv_s)$, and let $\H$ be a finitely generated subsemigroup of $\N\V$. 
Assume one of the following conditions. 
\begin{enumerate}
\item $\H$ is a normal affine semigroup. 
\item There exists $1\le i_1 <\dots <i_{t}\le s$ such that 
$\{\vv_1,\dots,\vv_s\}=\{\vv_{i_1},\dots,\vv_{i_t}\}$, and $\vv_{i_1},\dots,\vv_{i_t}$ are linearly independent. 
\end{enumerate}
Then $S^{(\H)}$ is Noetherian. 
\end{Proposition}
\begin{proof}
Set $\widetilde{\H}=\{\a\in\N^s \mid \V\cdot \a \in \H\}$. 
Recall that  $S^{(\H)}$ is Noetherian if and only if $\widetilde{\H}$ is finitely generated semigroup. 

(1) As $\H$ is normal, there exist sublattice $L\subset \Z^d$ and finitely generated rational cone $C\subset \Q^d$ such that $\H=L\cap C$. 
We regard $\V$ as a $\Q$-linear morphism from $\Q^s$ to $\Q^d$. 
Let $\widetilde{L}$ and $\widetilde{C}$ denote the pull-back of $L$ and $C$ under $\V$, respectively. 
Then $\widetilde{\H}=\widetilde{L}\cap \widetilde{C}\cap \N^s=\widetilde{L}\cap\Z^d\cap \widetilde{C}\cap \R_{\ge 0}^s$. 
Since $\widetilde{L}\cap\Z^d$ is a sublattice of $\Z^s$, and $\widetilde{C}\cap \R_{\ge 0}^s$ is a finitely generated rational cone, 
$\widetilde{\H}$ is also normal. 
Therefore, $\widetilde{\H}$ is finitely generated by Gordan's Lemma. 

(2) Note that $\Z\V=\Z\vv_{i_1}\oplus \dots \oplus \Z\vv_{i_t}$. 
Let $\{\w_1,\dots, \w_\ell\} \subset \N\V$ be a system of generators of $\H$. 
Since $\vv_{i_1},\dots,\vv_{i_t}$ are linearly independent, $\Fiber_\V(\w_j)\cap\N^s =\{\a\in \N^s \mid \A\cdot \a =\w_j\}$ 
is a subset of $\{\a\in \N^s\mid \abs{\a}=m_j\}$ for some positive integer $m_j>0$. 
Thus $\Fiber_\V(\w_j)\cap\N^s$ is a finite set. 
We claim that $\bigcup_{j=1}^\ell \Fiber_\V(\w_j)\cap\N^s$ is a finite system of generators of $\widetilde{\H}$. 

We denote by $\G$ the affine semigroup generated by $\bigcup_{j=1}^\ell \Fiber_\V(\w_j)\cap\N^s$. 
Then $\G\subset \widetilde{\H}$. 
For the converse inclusion, let $\a={}^t(a_1,\dots,a_s)\in \widetilde{\H}$. 
We will prove $\a\in\G$ by induction on $\abs{\a}$. 
It is trivial if $\abs{\a}=0$. 
There exist $b_1,\dots,b_\ell\in \N$ such that $\V\cdot \a=\sum_{j=1}^\ell b_j\w_j$. 
Without loss of generality, we may assume that $b_1\neq 0$. 
It is easy to show that there exists $\a'={}^t(a'_1,\dots,a'_s)\in \N^s$ such that $\V\cdot\a'=\w_1$ and $a'_i\le a_i$ for all $1\le i\le s$. 
Since $\V\cdot(\a-\a')=(b_1-1)\w_1+\sum_{j=2}^\ell b_j\w_j\in \H$, $\a-\a'\in \N^s$, and $\abs{\a-\a'}<\abs{\a}$, 
we have $\a-\a'\in \G$ by the hypothesis of induction. 
As $\a'\in \Fiber_\V(\w_1)\cap\N^s\subset \G$, and $\a=(\a-\a')+\a'$, we conclude that $\a\in \G$. 
\end{proof}
In the rest part of this paper, we consider only the case where $S^{(\H)}$ is Noetherian, and use the following notation. 
\begin{Notation}\label{notation}
Let $S=K[{y}_1, \dots, {y}_s]$, and $d>0$ a positive integer. 
We fix a $d\times s$ integer matrix $\V=(\vv_1,\dots,\vv_s)$ with the column vectors $\vv_1, \dots, \vv_s\in \Z^d$. 
We define a $\Z^d$-graded structure on $S$ by setting $\deg_{\Z^d}({y}_i)=\vv_i$. 
Then $\deg_{\Z^d} \y^\a=\V\cdot \a$ for $\a\in \N^s$, and $S=\bigoplus_{\vv \in \Z^d} S_\vv$ 
where $S_\vv$ is the $K$-vector space spanned by all monomials in $S$ of multi-degree $\vv$. 
Let $\H$ be a finitely generated subsemigroup of $\N\V$. 
We assume that the semigroup 
\[
\widetilde{\H}:=\{\a={}^t(a_1,\dots,a_s)\mid \V\cdot\a\in \H\} 
\]
is finitely generated. 
Let $\A_\H=\{\a^{(1)}, \dots, \a^{(r)}\}\subset \N^s$ be a system of generators  of $\widetilde{\H}$ as a semigroup. 
Then $S^{(\H)}=K[\A_\H]$. 
Let $R^{[\H]}=K[{x}_1, \dots, {x}_r]$ a polynomial ring over $K$, 
and $\phi_{\A_\H}:R^{[\H]}\to S$, $x_i\mapsto\y^{\a^{(i)}}$, the monomial homomorphism. 
\end{Notation}
Remark that $\A_\H$ is not always a configuration. 
\begin{Definition}
For $\vv \in \Z^d$, we define 
\[
C_\H(\vv)=\bigoplus_{\uu\in (-\vv +\H)\cap \Z^d}S_\uu, 
\]
a $\Z^d$-graded ${K[\A_\H]}$-submodule of $S$. 
Let $\Gamma_\H(\vv)$ be the minimal system of generators of $C_\H(\vv)$ as an ${K[\A_\H]}$-module consisting of monomials in $S$. 
\end{Definition}
If $S_\vv \neq 0$, then $C_\H(\vv)\cong f\cdot C_\H(\vv) \subset {K[\A_\H]}$ for any $0\neq f\in S_\vv$. 
Hence $C_\H(\vv)$ is isomorphic to an ideal of ${K[\A_\H]}$ up to shift of grading, in particular, finitely generated over ${K[\A_\H]}$. 
\begin{Lemma}\label{injection}
Let the notation be as in Notation \ref{notation}. 
Fix a weight vector $\w\in \N^s$ on $S$, and regard $S$ and $K[\A_\H]$ as $\N$-graded rings. 
Let $I$ be a $\Z^d$-graded ideal with $\Z^d$-homogeneous system of generators $F=\{f_1, \dots, f_\ell\}$ 
with $\deg_{\Z^d}(f_i)={{{\vv}}_i}\in \Z^d$. 
Then the following hold: 
\begin{enumerate}
\item 
$I \cap {K[\A_\H]}$ is generated by $\{\y^\a\cdot f_i\mid i\in [\ell], ~~\y^\a \in \Gamma_\H({{\vv}}_i) \}$. 
\item 
$\initial_{{K[\A_\H]}}(I\cap {K[\A_\H]})=\initial_\w(I)\cap {K[\A_\H]}$. 
\item If $F$ is a pseudo-Gr\"obner basis of $I$ with respect to $\w$, 
then 
\[
\{\y^\a\cdot f_i\mid i\in [\ell], ~~\y^\a \in \Gamma_\H({{\vv}}_i)\}
\]
is a pseudo-Gr\"obner basis of $I\cap {K[\A_\H]}$ in sense of Definition \ref{pGB}. 
\end{enumerate}
\end{Lemma}
\begin{proof}
(1) 
For $i\in [\ell]$ and $\y^\a \in \Gamma_\H({{\vv}}_i)$, 
$\y^\a\cdot f_i$ is a $\Z^d$-homogeneous element whose degree is in $\H$ by the definition of $ \Gamma_\H({{\vv}}_i)$, 
thus $\y^\a\cdot f_i\in K[\A]$. 
Let $J$ be the ideal of $K[\A_\H]$ generated by $\{\y^\a\cdot f_i\mid i\in [\ell], ~~\y^\a \in \Gamma_\H({{\vv}}_i) \}$. 
Then $J\subset I\cap {K[\A_\H]}$. 

For the converse inclusion, take a $\Z^d$-homogeneous element $g\in I\cap {K[\A_\H]}$, $\deg(g)={{\vv}}$, 
and write $g=\sum h_if_i$ where $h_i$'s are $\Z^d$-homogeneous elements with $\deg(h_if_i)={{\vv}}$. 
Since ${{\vv}}\in \H$, it holds that $h_if_i\in I\cap {K[\A_\H]}$, and thus $\deg_{\Z^d}(h_i)+{{\vv}}_i\in \H$. 
Hence $h_i\in C_\H({{\vv}}_i)$. Therefore it follows that $g\in J$. 

(2), (3) 
Assume that $\{f_1, \dots, f_\ell\}$ is a pseudo-Gr\"obner basis with respect to $\w$. 
Since $\initial_\w(I)$ is also $\Z^d$-graded ideal and $\initial_\w(f_i)\in S_{{{\vv}}_i}$, 
the contraction ideal $\initial_\w(I)\cap {K[\A_\H]}$ is generated by 
$\{\y^\a\cdot \initial_\w(f_i)\mid i\in [\ell], ~~\y^\a \in \Gamma_\H({{\vv}}_i)\}$. 
As $\y^\a\cdot \initial_\w(f_i)=\initial_\w(\y^\a\cdot f_i)$, 
we conclude $\initial_{{K[\A_\H]}}(I\cap {K[\A_\H]})=\initial_\w(I)\cap {K[\A_\H]}$. 
\end{proof}
\begin{Proposition}\label{=}
Let the notation as in Notation \ref{notation}. 
We denote by ${\w'}$ the weight vector $\phi_{\A_\H}^{*}\w=(\deg_\w\phi_{\A_\H}({x}_1), \dots, \deg_\w\phi_{\A_\H}({x}_r))$. 
Then  
\[
\phi_{\A_\H}^{-1}(\initial_\w(I))=\initial_{{\w'}}(\phi_{\A_\H}^{-1}(I)). 
\]
\end{Proposition}
\begin{proof}
We also denote by $\phi_{\A_\H}$ the surjection $R\to {K[\A_\H]}$. 
By Lemma \ref{injection}, 
\[
\initial_{{K[\A_\H]}}(I\cap {K[\A_\H]})=\initial_\w(I)\cap {K[\A_\H]}, 
\]
and by Lemma \ref{surjection} (2), 
\[
\phi_{\A_\H}^{-1}(\initial_{{K[\A_\H]}}(I\cap {K[\A_\H]}))=\initial_{{\w'}}(\phi_{\A_\H}^{-1}(I\cap {K[\A_\H]})). 
\]
Thus $\phi_{\A_\H}^{-1}(\initial_\w(I)\cap {K[\A_\H]})=\initial_{{\w'}}(\phi_{\A_\H}^{-1}(I\cap {K[\A_\H]}))$. 
Since  $\phi_{\A_\H}^{-1}(J\cap {K[\A_\H]})=\phi_{\A_\H}^{-1}(J)$ for an ideal $J\subset S$, 
we conclude that $\phi_{\A_\H}^{-1}(\initial_\w(I))=\initial_{{\w'}}(\phi_{\A_\H}^{-1}(I)). 
$. 
\end{proof}
\begin{Theorem}\label{main}
Let the notation be as in Notation \ref{notation}. 
Let $I$ be a $\Z^d$-graded ideal, 
$\prec$ a term order on $R^{[\H]}$, and $\w\in \N^s$ a weight vector of $S$ such that 
$\initial_\w(I)$ is a monomial ideal. 
We denote by ${\w'}$ the weight vector $\phi_{\A_\H}^{*}\w=(\deg_\w\phi_{\A_\H}({x}_1), \dots, \deg_\w\phi_{\A_\H}({x}_r))$. 
Then the following hold: 
\begin{enumerate}
\item 
$\delta(\initial_{\prec_{{\w'}}}(\phi_{\A_\H}^{-1}(I)))\le \max\{\delta(\initial_\w(I)), 
\delta(\initial_\prec(P_{\A_\H}))\}$. 
\item If both of $\initial_{\w}(I)$ and $\initial_\prec(P_{\A_\H})$ are generated by square-free monomials, 
then $\initial_{\prec_{{\w'}}}(\phi_{\A_\H}^{-1}(I))$ is generated by square-free monomials. 
\end{enumerate}
\end{Theorem}
\begin{proof}
The assertions follows from Theorem \ref{if=} and Proposition \ref{=}. 
\end{proof}
Note that Theorem \ref{main} holds true even if $\A_\H$ is not a configuration. 
\subsection{Pseudo-Gr\"obner bases and Gr\"obner bases of contraction ideals}
We will give a method to construct a pseudo Gr\"obner bases of $\phi_{\A_\H}^{-1}(I)$, 
and investigate when it become a Gr\"obner basis. First, we fix the notation in this subsection. 
\begin{Notation}\label{notation 2}
Let $S=K[{y}_1, \dots, {y}_s]$, $\A_\H$, $R^{[\H]}=K[{x}_1, \dots, {x}_r]$, 
and $\phi_{\A_\H}:R^{[\H]}\to S$ be as in Notation \ref{notation}. 
Let $\prec$ be a term order on $R^{[\H]}$, and $G_{\A_\H}$ 
the reduced Gr\"obner basis of $P_{\A_\H}$ with respect to $\prec$ consisting of binomials. 
Let $I\subset S$ be a $\Z^d$-graded ideal, and fix a weight vector $\w\in \N^s$ on $S$ such that $\initial_\w(I)$ is a monomial ideal. 
We denote by $\w'$ the weight vector $\phi_{\A_\H}^{*}\w=(\deg_\w\phi_{\A_\H}({x}_1), \dots, \deg_\w\phi_{\A_\H}({x}_r))$. 
\end{Notation}
\begin{Definition}\label{lift}
By Remark \ref{monomial basis}, 
for $0\neq q\in K[{\A_\H}]$, there is the unique polynomial $\tilde{q}\in R^{[\H]}$ 
such that $\phi_{\A_\H}(\tilde{q})=q$ and any term of $\tilde{q}$ is not in $\initial_\prec(P_{\A_\H})$. 
We define $\lift_\prec(q)=\tilde{q}$. 
For a subset $Q\subset K[{\A_\H}]$, we write 
\[
\lift_\prec(Q)=\{\lift_\prec(q)\mid q\in Q\}. 
\]
\end{Definition}
\begin{Remark}\label{subring}
\begin{enumerate}
\item For $q\in K[{\A_\H}]$, take a polynomial $p\in R^{[\H]}$ such that $\phi_{\A_\H}(p)=q$. 
Then $\lift_\prec(q)$ is the remainder of $p$ on division by $G_{\A_\H}$ with respect to $\prec$. 
\item 
If $u\in K[{\A_\H}]$ is a monomial, 
then $\lift_\prec(u)$ is a monomial such that $\deg_\w(u)=\deg_{\w'}(\lift_\prec(u))$  
since the remainder of a monomial on division by a $\w'$-homogeneous binomial ideal is a monomial with the same degree. 
Therefore, if $Q$ is a set of monomials, then so is $\lift_\prec(Q)$. 
\item 
Let $q\in K[{\A_\H}]\subset S$. 
If $\initial_\w(q)$ is a monomial, 
then $\initial_{{\w'}}(\lift_\prec(q))$ is also a monomial and $\deg_{\w'}(\lift_\prec(q))=\deg_\w(q)$ by (2). 
Furthermore, 
$\initial_{{\w'}}(\lift_\prec(q))=\lift_\prec(\initial_\w(q))$,  
and 
$\phi_{\A_\H}(\initial_{{\w'}}(\lift_\prec(q)))=\initial_\w(q)$. 
\item Since $R^{[\H]}/\Ker\phi_{\A_\H}\cong K[{\A_\H}]$ as $\N$-graded rings, 
for an ideal $J$ of $K[{\A_\H}]$ with a system of generators $Q$, we have 
$\phi_{\A_\H}^{-1}(J)=\langle \lift_\prec(Q) \rangle+\Ker \phi_{\A_\H}$. 
\end{enumerate}
\end{Remark}
\begin{Proposition}\label{lift of pGB}
Let $J$ be an ideal in $K[{\A_\H}]$ with a pseudo-Gr\"obner basis $Q=\{q_1, \dots, q_\ell\}$ 
(in the sense of Definition \ref{pGB} with a graded ring structure given by $\w$). 
Then $\lift_\prec(F)\cup G_{\A_\H}$ is a pseudo-Gr\"obner basis of $\phi_{\A_\H}^{-1}(J)$ with respect to ${\w'}$. 
\end{Proposition}
\begin{proof}
This easily follows from the above remarks. 
\end{proof}
Combining Proposition \ref{lift of pGB} and Lemma \ref{injection}, 
we can obtain a pseudo-Gr\"obner basis of $\phi_{\A_\H}^{-1}(I)$. 
\begin{Definition}
For a finite set $F=\{f_1, \dots, f_\ell\}\subset S$, $\deg_{\Z^d}f_i={\vv}_i$, we define 
\[
\Lift(F):=\lift_\prec(\{\y^\a\cdot f_i\mid i\in [\ell], ~~\y^\a \in \Gamma_\H({{\vv}}_i)\}) 
\]
\end{Definition}
\begin{Proposition}\label{lifted pGB}
Let the notation be as in Notation \ref{notation 2}. 
Let $F=\{f_1, \dots, f_\ell\}$ a pseudo-Gr\"obner basis of $I$ with respect to $\w$ consisting of $\Z^d$-homogeneous polynomials. 
Then the union $\Lift(F)\cup G_{\A_\H}$ 
is a pseudo-Gr\"obner basis of $\phi_{\A_\H}^{-1}(I)$ with respect to ${\w'}$. 
\end{Proposition}
Remark that $G_{\A_\H}\cup \Lift(F)$ is not always a Gr\"obner basis even if $I$ is a principal monomial ideal. 
\begin{Example}\label{pGB is not GB}
Let $S=K[y_1,y_2,y_3]$ be an $\N$-graded ring with $\deg(y_1)=\deg(y_2)=\deg(y_3)=1$. 
Let $\H=\{2 n\mid n\in \N\}\subset \N$. Then $\A_\H=\{{}^t(2,0),{}^t(1,1),{}^t(0,2)\}$, 
$R^{[\H]}=K[x_1,x_2,x_3]$, and $\phi_{\A_\H}: R^{[\H]}\to S$, $x_1\mapsto y_1^2$, $x_2\mapsto y_1y_2$, $x_3\mapsto y_2^2$. 
Let $\prec$ be the lexicographic order on $R^{[\H]}$ such that $x_1\prec x_2\prec x_3$. 
Then the reduced Gr\"obner basis $G_{\A_\H}$ of $P_{\A_\H}$ is $\{\underline{x_1x_3}-x_2^2\}$. 
Let $I=\langle y_2y_3^3\rangle$ and $F=\{y_2y_3^3\}$. 
Then $\Lift(F)=\{x_2x_3\}$ and $\phi_{\A_\H}^{-1}(I)=\langle x_2x_3, x_1x_3-x_2^2\rangle$. 
Let $\w$ be any weight vector on $S$. 
Since $\phi_{\A_\H}^{-1}(I)$ is a $\w'$-homogeneous ideal, 
$G_{\A_\H}\cup \Lift(F)=\{{x_2x_3}, {x_1x_3}-x_2^2\}$ is pseudo-Gr\"obner basis of $\phi_{\A_\H}^{-1}(I)$ with respect to $\w'$. 
However, $G_{\A_\H}\cup \Lift(F)$ is not a Gr\"obner basis of $\phi_{\A_\H}^{-1}(I)$. 
Recall that $G_{\A_\H}\cup M^{(\A)}_\prec(I)$ is a Gr\"obner basis of $\phi_{\A_\H}^{-1}(I)$ with respect to $\prec_{\w'}$ by Proposition \ref{monomial cases}. 
We have $M^{(\A)}_\prec(I)=\{{x_2x_3}, {x_2^3}\}$, 
and $G_{\A_\H}\cup M^{(\A)}_\prec(I)=\{{x_2x_3}, {x_1x_3}-x_2^2, {x_2^3}\}$ is a Gr\"obner basis of $\phi_{\A_\H}^{-1}(I)$ with respect to $\prec_{\w'}$. 
\end{Example}
In Example \ref{pGB is not GB}, the monomial ${x_2^3}$ is obtained by the S-polynomial $S({x_2x_3}, {x_1x_3}-x_2^2)$, 
and has degree $3$ which is strictly greater than $\deg({x_2x_3})=2$. 
We will give a sufficient condition for the pseudo-Gr\"oner basis constructed in Proposition \ref{lifted pGB}, 
to be a Gr\"obner basis. 
\begin{Proposition}\label{lifted GB}
Let the notation be as in Notation \ref{notation 2}. 
Assume that $\A_\H$ is a configuration. 
Suppose that $F=\{f_1, \dots, f_\ell\}$ is a Gr\"obner basis of $I$ with respect to $\w$. 
Let $L_i=L^{(\A)}_\prec(\initial_\w(f_i))$ and $M_i=M^{(\A)}_\prec(\initial_\w(f_i))$. 
Assume that for each $i$, there exists $\delta_i\in \N$ such that $\deg(u)=\delta_i$ for all $u\in M_i$. 
Then $G_{\A_\H}\cup \Lift(F)$ is a Gr\"obner basis of $\phi_{\A_\H}^{-1}(I)$ with respect to $\prec_{{\w'}}$. 
\end{Proposition}
\begin{proof}
Note that $P_{\A_\H}$ is a homogeneous ideal as $\A_\H$ is a configuration. 
As $G_{\A_\H}\cup \Lift(F)$ is a pseudo-Gr\"obner basis of $\phi_{\A_\H}^{-1}(I)$ and 
$G_{\A_\H}$ consists of ${\w'}$-homogeneous polynomials, 
the initial ideal $\initial_{{\w'}}(\phi_{\A_\H}^{-1}(I))$ is generated by 
$G_{\A_\H}\cup \{\initial_{{\w'}}(g)\mid g\in \Lift(F)\}$. 
Since 
\[
\initial_{\prec_{{\w'}}}(\phi_{\A_\H}^{-1}(I))
=\initial_{\prec}(\initial_{{\w'}}(\phi_{\A_\H}^{-1}(I))) 
\]
$G_{\A_\H}\cup \Lift(F)$ is a Gr\"obner basis of $\phi_{\A_\H}^{-1}(I)$ 
if and only if $G_{\A_\H}\cup \{\initial_{{\w'}}(g)\mid g\in \Lift(F)\}$ is a Gr\"obner basis of 
$\phi_{\A_\H}^{-1}(\initial_\w(I))=\initial_{\w'}(\phi_{\A_\H}^{-1}(I))$. 
By Remark \ref{subring} (3), it follows that  
\[
\{\initial_{{\w'}}(g)\mid g\in \Lift(F)\}=\Lift(\{\initial_\w(f)\mid f\in F\}). 
\]
Thus it is enough to show that 
\[
G_{\A_\H}\cup \Lift(\{\initial_\w(f)\mid f\in F\}) 
\]
is a Gr\"obner basis of $\phi_{\A_\H}^{-1}(\initial_\w(I))$ with respect to $\prec_{{\w'}}$. 
Since $\{\initial_\w(f)\mid f\in F\}$ is a system of generators of the monomial ideal $\initial_\w(I)$, 
it is enough to prove this theorem of $\initial_\w(I)$. 
Thus we may, and do assume that $I$ is a monomial ideal and 
$F=\{f_1, \dots, f_\ell\}$ is the minimal system of monomial generators of $I$. 
Then $G_{\A_\H}\cup(\bigcup_{i=1}^\ell M_i)$ is a Gr\"obner basis of $\phi_{\A_\H}^{-1}(I)$ 
by Proposition \ref{monomial cases}. 

Note that $\Lift(F)$ is a set of monomials, and $G_{\A_\H}\cup \Lift(F)$ is a system of generators of $\phi_{\A_\H}^{-1}(I)$. 
We will prove that $G_{\A_\H}\cup \Lift(F)$ is a Gr\"obner basis of $\phi_{\A_\H}^{-1}(I)$ with respect to $\prec$ using 
Buchberger's criterion. 
It is enough to show that the remainder of the S-polynomial  $S(u, g)$ on division by $G_{\A_\H}\cup \Lift(F)$ is zero for all 
$u\in \Lift(F)$ and $g\in G_{\A_\H}$. 
Let $u\in \Lift(F)$. 
Then $u\in L_i$ for some $i$, and thus $\deg(u)\ge \delta_i$. 
For any $g\in G_{\A_\H}$, as $u\not\in \initial_\prec(P_{\A_\H})$, 
it follows that $u\neq \initial_\prec(g)$ and thus 
the degree of the S-polynomial $S(u, g)$ is strictly greater than $\delta_i$ . 
Let $u'$ be a remainder of $S(u, g)$ on division by $G_{\A_\H}\cup \Lift(F)$. 
Since $G_{\A_\H}\cup \Lift(F)$ is a set of homogeneous binomials and monomials, 
$u'$ is zero or a monomial in $L_i$ of degree $\deg (S(u, g))>\delta_i$. 
Hence $\initial_\prec(u')=u'$ is zero or 
not a member of the minimal system of monomial generators of $\initial_{\w'}(\phi_{\A_\H}^{-1}(I))$. 
If $u'\neq 0$ for some $u\in \Lift(F)$, this contradicts to the next lemma. 
\end{proof}
\begin{Lemma}\label{hGB criterion}
Let $I\subset K[{x}_1, \dots, {x}_r]$ be a homogeneous ideal with a homogeneous system of generators $G=\{g_1, \dots, g_\ell\}$. 
If $G$ is not a  a Gr\"obner basis of $I$, then there exist $1\le i<j \le \ell$ such that 
the initial of $\overline{S(g_i, g_j)}^G$ is a member of the minimal system of monomial generators of 
$\initial_\prec(I)$. 
\end{Lemma}
\begin{proof}
First, note that if $\overline{S(g_i, g_j)}^G\neq 0$, 
then the degree of $\overline{S(g_i, g_j)}^G$ is not less than the degrees of $g_i$ and $g_j$. 
Assume, to the contrary, the initial of $\overline{S(g_i, g_j)}^G$ is zero 
or not a member of the minimal system of monomial generators of $\initial_\prec(I)$ for all $1\le i<j\le \ell$. 
Let $F=\{\x^{\b^{(1)}}, \dots, \x^{\b^{(m)}}\}$ be the minimal system of monomial generators of $\initial_\prec(I)$. 
Assume that $G$ is not a Gr\"obner basis. 
We may assume that $\x^{\b^{(1)}}$ is the monomial of minimal degree among monomials in $F$ which are not in 
$ \langle \initial_\prec(g)\mid g\in G \rangle$. 
Let $G'\subset I$ be a finite subset of $I$ 
such that $G\cup G'$ is a minimal Gr\"obner basis of $I$ computed from $G$ by Buchberger's algorithm. 
Then there exists $h\in G'\backslash G$ such that $\x^{\b^{(1)}}=\initial_\prec(h)$. 
By the procedure of Buchberger's algorithm, 
$\deg(\x^{\b^{(1)}})=\deg h\ge \deg \overline{S(g_i, g_j)}^G$ 
for some $1\le i<j\le \ell$ with $\overline{S(g_i, g_j)}^G\neq 0$. 
Let $\x^{\b^{(k)}}\in F$ such that $\x^{\b^{(k)}}$ divides the initial of $\overline{S(g_i, g_j)}^G$. 
By the assumption, the initial of $\overline{S(g_i, g_j)}^G$ does not coincide with $\x^{\b^{(k)}}$, 
and thus the degree of $\x^{\b^{(k)}}$ is strictly less than $\overline{S(g_i, g_j)}^G$. 
Therefore $\deg(\x^{\b^{(1)}})>\deg(\x^{\b^{(k)}})$. 
Since any term of $\overline{S(g_i, g_j)}^G$ is not in $\langle \initial_\prec(g)\mid g\in G \rangle$, 
we have $\x^{\b^{(k)}}\not\in \langle \initial_\prec(g)\mid g\in G \rangle$ . 
This is a contradiction. 
\end{proof}
\section{Applications}
We will present some applications of  Theorem \ref{main}. 
In all examples presented in this section, the grading of the polynomial ring of $S$ satisfies the condition of Proposition \ref{noetherian condition} (2).  
\subsection{Veronese configurations}
Let $S=K[{y}_1, \dots, {y}_s]=\bigoplus_{i\in \N}S_i$ be an $\N$-graded ring with $\deg({y}_i)=1$ for all $i$. 
Let ${d}$ be a positive integer, and 
\[
{\A_d}=\{\a={}^t(a_1, \dots, a_s)\in \N^s\mid \abs{\a}={d}\}
\]
the Veronese configuration. 
Then $S^{(\N\cdot d)}=K[\A_d]$ is the $d$-the Veronese subring of $S$. 
Let $R=K[{x}_\a\mid \a\in {\A_d}]$ be a polynomial ring, and $\phi_{\A_d}: R\to S$ (${x}_\a\mapsto \y^\a$) the monomial homomorphism. 
It is known that there exist a lexicographic order on $R$ such that 
$\initial_\prec(P_{\A_d})$ is generated by square-free monomial of degree two (\cite{Emanuela}). 
\begin{Theorem}
Let $I\subset S$ be a homogeneous ideal, $\w$ a weight vector on $S$ such that $\initial_\w(I)$ is a monomial ideal, 
and $\prec$ a term order on $R$ such that $\initial_\prec(P_{\A_d})$ is generated by square-free monomial of degree two. 
We denote the weight vector $\phi_{\A_d}^{*}\w$ by $\w'$. 
Then the following hold: 
\begin{enumerate}
\item 
$\delta(\initial_{\prec_{{\w'}}}(\phi_{\A_d}^{-1}(I))
\le \max\{2, \delta(\initial_{\w}(I))\}$. 
\item If $\initial_{\w}(I)$ is generated by square-free monomials, 
then $\initial_{\prec_{\w'}}(\phi_{\A_d}^{-1}(I))$ is generated by square-free monomials. 
\end{enumerate}
\end{Theorem}
\begin{proof}
The assertion immediately follows from Theorem \ref{main}. 
\end{proof}
Eisenbud--Reeves--Totaro proved in \cite{ERT} that 
if $K$ is an infinite field, the coordinates ${y}_1, \dots, {y}_s$ of $S$ are generic, 
and $\prec$ is a certain reversed lexicographic order, 
then it holds that 
$\delta(\initial_{\prec_{{\w'}}}(\phi_{\A_d}^{-1}(I))\le \max\{2, \delta(\initial_{\w}(I))/{d}\}$. 
\subsection{Toric fiber products}
We recall toric fiber products defined in \cite{Sullivant}. 
Let $s_1, \dots, s_d$, $t_1, \dots, t_d$ and $d$ be positive integers, and  let 
\[
S_1=K[\y]=K\bigl[y^{(i)}_j\mid i\in [d], j\in [s_i]\bigr], ~~ S_2=K[\z]=K\bigl[z^{(i)}_k\mid i\in [d], k\in [t_i]\bigr], 
\] 
be $\Z^d$-graded polynomial rings regarded with 
\[
\deg(y^{(i)}_j)=\deg(z^{(i)}_k)=\e_i 
\]
for $i\in [d], j\in [s_i], k\in[t_i]$ where $\e_i$ is the vector with unity in the $i$-th position and zeros elsewhere. 
Then 
\[
S:=S_1\otimes_K S_2\cong K[y^{(i)}_j, z^{(i)}_k\mid i\in [d], j\in [s_i], k\in [t_i]\bigr] 
\]
carries a $\Z^d\times\Z^d$-graded ring structure by setting 
\[
\deg_S(y^{(i)}_j)=(\e_i, \mathbf{0}), ~~\deg_S(z^{(i)}_k)=(\mathbf{0}, \e_i)
\]
for $i\in [d], j\in [s_i], k\in[t_i]$ in $S$. 
Let 
\[
\Delta=\{(\vv, \vv)\mid \vv\in \Z^d\}\subset \Z^d\times\Z^d 
\]
be the diagonal subsemigroup of $\Z^d\times\Z^d$. 
Since $\Delta$ is generated by $\{(\e_i, \e_i)\mid i\in[d]\}$, we have 
\[
S^{(\Delta)}\cong K[y^{(i)}_jz^{(i)}_k\mid i\in [d], j\in [s_i], k\in [t_i]\bigr]. 
\]
Let $R=K\bigl[x^{(i)}_{jk}\mid i\in [d], j\in [s_i], k\in [t_i]\bigr]$ be a polynomial ring, 
and $\phi:R\to S$ the monomial homomorphism $\phi(x^{(i)}_{jk})=y^{(i)}_jz^{(i)}_k$. 

Let $I_1\subset S_1$ and $I_2\subset S_2$ be $\Z^d$-graded ideals, 
and denote $(I_1\otimes S_2)+(S_1\otimes I_2)\subset S$ simply by $I_1+I_2$. 
The ideal 
\[
I_1\times_{\Z^d}I_2:=\phi^{-1}(I_1+I_2)
\] 
is called the {\it toric fiber product} of $I_1$ and $I_2$. 
Originally, the assumptions in \cite{Sullivant} are that $\deg(y^{(i)}_j)=\deg(z^{(i)}_k)=\a^{(i)}\in \Z^d$ with 
$\a^{(1)}, \dots, \a^{(d)}$ linearly independent, and $I_1$ and $I_2$ are $\Z^d$-graded ideals, 
which are equivalent to ours. 

Let $\w_1$ and $\w_2$ weight vectors of $S_1$ and $S_2$ such that 
$\initial_{\w_1}(I_1)$ and $\initial_{\w_2}(I_2)$ are monomial ideals, 
and set $\w=(\w_1, \w_2)$, the weight oder of $S$. 
Let $G_1$ and $G_2$ be Gr\"obner bases of $I_1$ and $I_2$ with respect to $\w_1$ and $\w_2$ respectively. 
\begin{Theorem}\label{toric fiber}
Let the notation be as above. 
Let $\prec$ be the lexicographic term order on $R$ such that
$x^{(i_1)}_{j_1k_1}\prec x^{(i_2)}_{j_2k_2}$ 
if $i_1 < i_2$ or $i_1 = i_2$ and $j_1 < j_2$ or $i_1 = i_2$ and $j_1 = j_2$ and $k_1 > k_2$. 
Then the following hold: 
\begin{enumerate}
\item 
$\delta(\initial_{\prec_{\phi^{*}\w}}(I_1\times_{\Z^d}I_2))
\le \max\{2, \delta(\initial_{\w_1}(I_1)), \delta(\initial_{\w_1}(I_2))\}$. 
\item 
If both of $\initial_{\w_1}(I_1)$ and $\initial_{\w_2}(I_2)$ are generated by square-free monomials, 
then $\initial_{\prec_{\phi^*\w}}(I_1\times_{\Z^d}I_2)$ is generated by square-free monomials. 
\end{enumerate}
\end{Theorem}
\begin{proof}
By \cite{Sullivant} Proposition 2.6, the Gr\"obner basis of $\Ker \phi$ with respect to $\prec$ is 
\[
\{ \underline{x^{(i)}_{j_1 k_2}x^{(i)}_{j_2k_1}}-x^{(i)}_{j_1k_1}x^{(i)}_{j_2k_2}\mid 
i\in [d], ~1\le j_1<j_2 \le s_i, ~1\le k_1<k_2\le t_i\} 
\]
where underlined terms are initial. 
Since $G=G_1\cup G_2$ is a Gr\"obner basis of $I_1+I_2$ with respect to $\w$, we have 
\[
\delta(\initial_\w(I_1+I_2))=\max\{\delta(\initial_{\w_1}(I_1)), \delta(\initial_{\w_1}(I_2))\}. 
\]
As $I_1+I_2$ is a $\Z^d\times \Z^d$-graded ideal, 
the assertions follow from Theorem \ref{main}. 
\end{proof}
In case of toric fiber product, the pseudo-Gr\"obner basis constructed as in Proposition \ref{lifted pGB} 
from Gr\"obner basis of $I_1+I_2$ is a Gr\"obner basis of $I_1\times_{\Z^d}I_2$. 
This is mentioned in \cite{Sullivant} Corollary 2.10, but the proof contains a gap 
(the author claims that the pseudo-Gr\"obner basis is a Gr\"obner basis without proof). 
The proof of Theorem \ref{toric fiber} (1) by Sullivant (\cite{Sullivant} Corollary 2.11) uses this fact. 
We give a correct proof here. 
\begin{Theorem}
Let $G_1$ and $G_2$ be Gr\"obner bases of $I_1$ and $I_2$ with respect to weight vectors $\w_1$ and $\w_2$ respectively, 
and set $\w=(\w_1, \w_2)$. 
Then the pseudo-Gr\"obner basis of $I_1\times_{\Z^d}I_2$ constructed from $G:=G_1\cup G_2$ as in Proposition \ref{lifted pGB} 
is a Gr\"obner basis of $I_1\times_{\Z^d}I_2$ with respect to $\prec_{\phi^{*}\w}$. 
\end{Theorem}
\begin{proof}
Let $g\in G$. 
We denote by $L_g$ a monomial ideal generated by all monomials in $\phi^{-1}(\initial_\w(g))\backslash {\Ker \phi}$. 
Let $M_g$ be the minimal system of monomial generators of $L_g$. 
By Proposition \ref{lifted GB}, it is enough to show that $\deg(u)=\deg(g)$ for all $u\in M_g$ to prove this theorem. 
Since $g\in S_1=K[\y]$ or $g\in S_2=K[\z]$, 
we may, and do assume that $g\in S_1$. 
Set $\initial_\w(g)=\y^\a$. 

Let $\x^\b\in M_g$. Then $\y^\a$ divide $\phi(\x^\a)$. 
Since the degree of $\phi(\x^\b)$ in $\y$ is the same as $\deg(\x^\b)$ by the definition of $\phi$, 
we have $\deg(\x^\a)\ge \deg(\y^\a)$. 
By Lemma \ref{pull-back of monomial ideals} (1), 
it holds that $\deg(\x^\a)\le \deg(\y^\a)$. 
Hence we conclude $\deg(\x^\b)\deg(\y^\a)= \deg(g)$. 
\end{proof}
\subsection{Generalized nested configurations}
Let $d$ and $\mu$ positive integers, 
and take $\lambda_i\in \N$ for $i\in[d]$. 
Let $\A$ be a configuration of $\N^d\subset\bigoplus_{i=1}^{d} \Z \e_i$, 
where $\e_i$ is the vector with unity in the $i$-th position and zeros elsewhere. 
Let \[
\B_{i} = \{\b^{(i)}_1, \dots, \b^{(i)}_{\lambda _i}\}\subset\N^\mu
\]
be a configuration of $\N^\mu$ for $i = 1, 2, \dots, d$. 
The {\it (generalized) nested configuration} arising from $\A$ and $\B_{1}, \dots, \B_{d}$ is the configuration 
\[
\A[\B_{1}, \dots, \B_{d}]:= 
\{\b^{(i_1)}_{j_1}+ \cdots +\b^{(i_r )}_{j_r}\mid 
1\le r \in \N, \e_{i_1}+\cdots + \e_{i_r}\in \A, ~ j_k\in [\lambda_{i_k}], ~ i_k\in [d]\}. 
\]
Originally, Aoki--Hibi--Ohsugi--Takemura (\cite{AHOT}) define nested configurations in case where 
there exists $0<\mu_1, \dots, \mu_d\in \N$ such that $\N^\mu=\N^{\mu_1}\times\dots \times \N^{\mu_d}$ and $\B_{i}\subset \N^{\mu_i}$. 
For $1\le i\le d$, let $\E_{i}=\{\e^{(i)}_1, \dots, \e^{(i)}_{\lambda_{i}}\}$ be the configuration of $\bigoplus_{j=1}^{\lambda_i} \Z \e^{(i)}_j$. 
Then $\E_{1}\cup \dots \cup \E_{d}$ is a configuration of $\bigoplus_{i=1}^{d}\bigoplus_{j=1}^{\lambda_i} \Z \e^{(i)}_j$. 
Let 
\[
S=K\bigl[\E_{1}\cup \dots\cup \E_{d}\bigr]\cong K\bigl[y^{(i)}_j\mid i\in[d], ~j\in[\lambda_i]\bigr]
\] 
be the $\N^{d}$-graded polynomial ring with $\deg_{\N^d} y^{(i)}_j=\e_i$. 
Then 
\[
S^{(\N\A)}=K\bigl[\A[\E_{1}, \dots, \E_{d}]\bigr]. 
\]
\begin{Example}
Let $S=K\bigl[y_1,y_2,y_3,y_4,y_5\bigr]$ be $\N^3$-graded ring with 
\[
\deg_{\N^3} y_1=\deg_{\N^3} y_2=\e_1, ~~\deg_{\N^3} y_3=\deg_{\N^3} y_4=\e_2, ~~\deg_{\N^3} y_5=\e_3. 
\]
Let $\A=\left(
	\begin{array}{ccc}
	1 & 1 & 0 \\
	1 & 0 & 2 \\
	1 & 2 & 1 \\
	\end{array}
\right)$. 
Then 
\[
S^{(\N\A)}=K[
y_1y_3y_5, y_1y_4y_5, y_2y_3y_5, y_2y_4y_5, \hspace{1mm}
y_1y_5^2, y_2y_5^2, \hspace{1mm} 
y_3^2y_5, y_3y_4y_5, y_4^2y_5
]. 
\]
\end{Example}
\begin{Theorem}[\cite{HO} Theorem 2.5]
Let the notation as above. Then the following holds: 
\begin{enumerate}
\item If $P_\A$ admit initial ideals of degree at most $m$, then so is $P_{\A[\E_{1}, \dots, \E_{d}]}$. 
\item If $P_\A$ admit square-free initial ideals, then so is $P_{\A[\E_{1}, \dots, \E_{d}]}$. 
\end{enumerate}
\end{Theorem}
\begin{Theorem}\label{nested1}
Let $S=K[\y]=K\bigl[y^{(i)}_j\mid i\in[d], ~j\in[\lambda_i]\bigr]$ be $\N^{d}$-graded polynomial ring with $\deg_{\N^d} y^{(i)}_j=\e_i$ as above. 
Let $R=K\bigl[{x}_\a\mid \a\in \A[\E_{1}, \dots, \E_{d}]\bigr]$ be a polynomial ring, and set $\phi_{\A[\E_{1}, \dots, \E_{d}]}: R\to S$, $x_\a\mapsto\y^\a$. 
Let $I\subset S$ be an $\N^{d}$-graded ideal. 
If $I$ and $P_\A$ admit quadratic Gr\"obner bases with respect to some term orders, 
then so does $\phi_{\A[\E_{1}, \dots, \E_{d}]}^{-1}(I)$. 
\end{Theorem}
\begin{proof}
If $P_\A$ admits a quadratic initial ideal, then so does $P_{\A[\E_{1}, \dots, \E_{d}]}$ (\cite{AHOT} Theorem 3.6). 
Therefore the assertion follows from Theorem \ref{main}. 
\end{proof}
\begin{Theorem}\label{generalized nested}
Let $d>0$ and $\lambda_i\in \N$ for $i\in [d]$ be integers, $\Z^d=\bigoplus_{i=1}^d \Z\e_i$ a free $\Z$-module of rank $d$ 
with a basis $\e_1, \dots, \e_d$, and 
$S=K\bigl[y^{(i)}_j\mid i\in [d], ~j\in[\lambda_i]\bigr]$ a $\Z^d$-graded polynomial ring with $\deg y^{(i)}_j=\e_i$ for $i \in [d]$, $j\in[\lambda_i]$. 
Let $\A\subset\N^d= \bigoplus_{i=1}^d\N \e_i$ be a configuration. 
Let $\V$ be a $\nu\times \mu$ integer matrix, and $\w_1,\dots,\w_d\in \Z^\nu$ linearly independent vectors. 
For $i\in[d]$, we fix a finite set 
\[
\B_i = \bigl\{\b^{(i)}_j\mid j\in[\lambda _i]\}\subset \Fiber_\V(\w_i)=\{\b \in \Z^\mu\mid \V\cdot\b=\w_i\bigr\}\subset\Z^\mu. 
\]
We set $\B=\B_1\cup\dots\cup\B_d$. 
Then $\B$ is a configuration of $\Z^\mu$, and the following hold. 
\begin{enumerate}
\item If both of $P_\B$ and $P_\A$ admit initial ideals of degree at most $m$, 
then so is $P_{\A[\B_1,\dots,\B_d]}$. 
\item If both of $P_\B$ and $P_\A$ admit square-free initial ideals, 
then so is $P_{\A[\B_1,\dots,\B_d]}$. 
\end{enumerate}
\end{Theorem}
\begin{proof}
Let $R=K\bigl[x_\a\mid \a\in \A[\E_{1}, \dots, \E_{d}] \bigr]$, $S=K\bigl[y^{(i)}_j\mid i\in[d], ~j\in[\lambda_i]\bigr]$ be polynomial rings, 
and set $\phi_{\A[\E_{1}, \dots, \E_{d}]}:R\to S$, $x_\a\mapsto \y^\a$, and $\phi_\B: S\to K[\z^{\pm 1}]=K[z_1^{\pm 1},\dots, z_\mu^{\pm 1}]$, 
$y^{(i)}_j\mapsto \z^{\b^{(i)}_j}$. 
Then $\phi_\B \circ \phi_{\A[\E_{1}, \dots, \E_{d}]}=\phi_{\A[\B_{1}, \dots, \B_{d}]}$. 
Since  $\w_1,\dots,\w_d\in \Z^\mu$ are linearly independent, $P_\B=\Ker\phi_\B$ is a $\Z^d$-graded ideal. 
Applying Theorem \ref{nested1} to the  $\Z^d$-graded ideal $P_{\B}$, we conclude the assertion. 
\end{proof}
\mbox{}\\
\noindent
{\bf Acknowledgement.} This research was supported by JST, CREST.


\begin{thebibliography}{99}
\bibitem{AHOT}
S. Aoki, T. Hibi, H. Ohsugi, A. Takemura, Gr\"obner bases of nested configurations, J. Algebra {\bf 320} (2008), no. 6, 2583--2593 
\bibitem{BG}
W. Bruns, J. Gubeladze, \textit{Polytopes, rings, and K-theory}, Springer Monographs in Mathematics. Springer, Dordrecht, 2009. 
\bibitem{Cox1} 
D. Cox, J. Little, D. O'Shea, \textit{Ideals, Varieties and Algorithms, } Springer-Verlag, 1992.
\bibitem{Cox2} 
D. Cox, J. Little, D. O'Shea, \textit{Using Algebraic Geometry, } Springer-Verlag, 1998. 
\bibitem{DS}
P. Diaconis, B. Sturmfels, Algebraic algorithms for sampling from conditional distributions, 
Ann. Statist. {\bf 26}, no. 1, 363--397, 1998. 
\bibitem{ERT}
D. Eisenbud, A. Reeves, B. Totaro, Initial ideals, Veronese subrings, and rates of algebras, Adv. Math. {\bf 109} (2) (1994), 168--187. 
\bibitem{Emanuela} 
D. Emanuela, Toric rings generated by special stable sets of monomials, 
Math. Nachr. {\bf 203} (1999), 31--45. 
\bibitem{HO}
T. Hibi, H. Ohsugi, Toric rings and ideals of nested configurations, J. Commut. Algebra {\bf 2} (2010), no. 2, 187--208.. 
\bibitem{Sturmfels} 
B. Sturmfels, \textit{Gr\"obner Bases and Convex Polytopes}, Univ. Lecture Ser., vol. 8, Amer.Math. Soc., Providence, 1996.
\bibitem{Sullivant} 
S. Sullivant, Toric fiber products, J. Algebra {\bf 316} (2007), no. 2, 560--577. 
\end{thebibliography}
\end{document}